\documentclass[12pt,letterpaper]{amsart}

\newcommand{\indentalign}{\hspace{0.3in}&\hspace{-0.3in}}
\newcommand{\la}{\langle}
\newcommand{\ra}{\rangle}

\newcommand{\defeq}{\stackrel{\rm{def}}{=}}

\usepackage{amssymb}
\usepackage{amsthm}
\usepackage{amsxtra}
\newtheorem{theorem}{Theorem}
\newtheorem{proposition}[theorem]{Proposition}
\newtheorem{lemma}[theorem]{Lemma}
\newtheorem{corollary}[theorem]{Corollary}

\theoremstyle{remark}

\numberwithin{equation}{section}
\numberwithin{theorem}{section}

\title[Low regularity bounds for mKdV] {Low regularity 
\textit{a priori} bounds for the modified Korteweg-de Vries equation}

\author{Michael Christ}
\address{University of California, Berkeley}
\email{mchrist@math.berkeley.edu} 
\author{Justin Holmer}
\address{Brown University}
\email{holmer@math.brown.edu}
\author{Daniel Tataru}
\address{University of California, Berkeley}
\email{tataru@math.berkeley.edu} 
\begin{document}

\maketitle

\begin{abstract}
  We study the local well-posedness in the Sobolev space
  $H^s(\mathbb{R})$ for the modified Korteweg-de Vries (mKdV) equation
  $\partial_t u + \partial_x^3 u \pm \partial_x u^3=0$ on
  $\mathbb{R}$.  Kenig-Ponce-Vega \cite{KPV2} and
  Christ-Colliander-Tao \cite{CCT1} established that the
  data-to-solution map fails to be uniformly continuous on a fixed
  ball in $H^s(\mathbb{R})$ when $s<\frac14$.  In spite of this, we
  establish that for $-\frac18 < s <\frac14$, the solution satisfies
  global in time $H^s(\mathbb{R})$ bounds which depend only on the
  time and on the $H^s(\mathbb{R})$ norm of the initial data.  This
  result is weaker than global well-posedness, as we have no control on
  differences of solutions.  Our proof is modeled on recent work by
  Christ-Colliander-Tao \cite{CCT2} and Koch-Tataru \cite{KT}
  employing a version of Bourgain's Fourier restriction spaces adapted
  to time intervals whose length depends on the spatial frequency.
\end{abstract}

\section{Introduction}

We study the well-posedness of the initial-value problem for the
modified Korteweg-de Vries (mKdV) on $\mathbb{R}$:
\begin{equation}
  \label{E:mKdV}
  \partial_tu + \partial_x^3 u \pm  \partial_x u^3 =0, \qquad u(0) = u_0
\end{equation}
where $u=u(x,t)\in \mathbb{R}$ with $(x,t)\in \mathbb{R}^{1+1}$.  This
equation has scaling 
\[
u(x,t) \mapsto \lambda u(\lambda x,\lambda^3t)
\]
and the scale invariant homogeneous Sobolev norm is $\dot
H^{-\frac12}$.  The equation is globally well-posed in $H^s$ for
$s\geq \frac14$.  Specifically, given initial data in $H^s$, a
solution exists in $C([0,+\infty); H^s)\cap X$, where $X$ is a certain
auxiliary function space; this solution is unique among all solutions
that reside in this function class; and for any $T>0$, the
data-to-solution map from a fixed ball in $H^s$ to $C([0,T];H^s)$ is
uniformly continuous.  The local result was proved by 
Kenig-Ponce-Vega \cite{KPV1} by the contraction method in a 
function space where
several dispersive estimates for the linear flow hold.  An alternate
proof in the setting of the Fourier restriction norm spaces was given
later in Tao \cite{Tao}.  Colliander-Keel-Staffilani-Takaoka-Tao
\cite{CKSTT} proved that this local solution extends to a global
solution by studying the almost conservation of of the norm of a high
frequency-damped copy of the solution (the $I$-method).  On the other
hand, for $s<\frac14$, \eqref{E:mKdV} on $\mathbb{R}$ is ill-posed in
the sense that the data-to-solution map fails to be \emph{uniformly}
continuous on a fixed ball in $H^s$.  This was established by
Kenig-Ponce-Vega \cite{KPV2} for the focusing equation ($+$ sign in
front of the nonlinearity; Theorem 1.3 on p. 623 of their paper), and
by Christ-Colliander-Tao \cite{CCT1} for the defocusing equation ($-$
sign in front of the nonlinearity; Theorem 4 on p. 1240 of their
paper)\footnote{The proof given by \cite{CCT1} holds for $-\frac14 <
  s< \frac14$, but the authors remark that the restriction to
  $s>-\frac14$ is likely an artifact of their method.}.  This leaves
open the question as to whether or not there is a well-posedness
result for $s<\frac14$ giving only the continuity (as opposed to
\emph{uniform} continuity) of the data-to-solution map.  One result in
this direction is Kato \cite{Kato}, where global weak solutions for
$s=0$ are constructed.  We will here prove another result in this
direction, giving an \emph{a priori} bound in $H^s$ for
$-\frac18<s<\frac14$ in terms of the $H^s$ norm of the initial data
but establishing no continuity.
Our method is analogous to that in Christ-Colliander-Tao \cite{CCT2}
and Koch-Tataru \cite{KT} dealing with the nonlinear Schr\"odinger
equation (NLS) on $\mathbb{R}$. The related problem for the 
mKdV equation was considered by Liu~\cite{bao}.

\begin{theorem}
  \label{T:main}
  Let  $-\frac18<s<\frac14$.  Then for any $R > 0$ and $T > 0$ there
  exists\footnote{The proof actually yields $C = \max \{ 1,
    R^{-\frac{8s}{1+8s}} T^{-\frac{s}{1+8s}}\}$ but this is very likely 
nonoptimal.}
 $C=C(R,T)>0$ so that for any initial data $u_0\in
  \mathcal{S}$ satisfying 
\[
\|u_0\|_{H^s} \leq R \, ,
\]
the unique solution $u\in C([0,T]; \mathcal{S})$ to \eqref{E:mKdV}
(focusing or defocusing) satisfies
\[
\|u\|_{L_{[0,T]}^\infty H_x^s} \leq C \|u_0\|_{H^s} \, .
\]
\end{theorem}
We note that our proof also applies for $s = -\frac18$, but with a 
$C$ which depends on the full $H^{-\frac18}$ frequency envelope
of $u$. This dependence is likely nonoptimal, and it would simplify
once the $-1/8$ threshold is crossed.

We also note that in the process of establishing the above result 
we also prove that the solutions belong to a smaller space $X^s$ 
defined later in the paper. 

An easy consequence of our result is the existence of weak solutions
for $H^s$ data:

\begin{corollary}
Given any initial data $u_0 \in H^s$, there exists a global solution 
$u$ to \eqref{E:mKdV} which solves the equation in the sense of distributions
and satisfies 
\[
\|u(t)\|_{H^s} \lesssim C(t, \|u_0\|_{H^s})
\]
with $C$ as in the theorem above.
\end{corollary}

The weak solution is constructed as a weak limit of strong solutions.
The uniform local $H^s$ bound does not suffice in order to verify that 
the equation is verified in the sense of distributions. Instead, this is true
due  the  uniform $X^s$ bound, which is also implicit in the construction.
We refer to these solutions as weak solutions as we currently do not
have any uniqueness or continuous dependence result in $H^s$ for 
$s < -\frac14$.

Currently the analogous problem for the periodic mKdV (\eqref{E:mKdV}
with $(x,t)\in \mathbb{T}\times \mathbb{R}$) is better understood.
The threshold of $s=\frac14$ for mKdV on $\mathbb{R}$ is replaced by
$s=\frac12$ for mKdV on $\mathbb{T}$.  Kappeler-Topalov \cite{KapTop}
construct, via inverse scattering theory, global solutions in $L^2$.
Tsutsumi-Takaoka \cite{TT} construct solutions for data in $H^s$ for
$\frac38<s<\frac12$ via Fourier restriction norm estimates and a
nonlinear ansatz.  Both of these results assert the continuity of the
data-to-solution map.

Regarding our result, we believe that in principle, by adding another
correction term (or maybe more) to the modified energy in \S
\ref{S:energy}, we could improve the lower threshold to $s\geq
-\frac16$ since the trilinear $\ell^2U_A^{s,2}$ estimate in \S
\ref{S:trilinear} is valid down to this threshold.  It seems that to
push to $s<-\frac16$ would require a better understanding of
``diagonal'' or ``resonant'' frequency interactions.  We do not know
if there is any significance to the number $s=-\frac16$ in regard to
the actual behavior of solutions or whether it is just an artifact of
our method.

An outline of the paper is as follows.  In \S\ref{S:norms}, we define
the function spaces employed in the analysis.  We use the $U^p$ and $V^p$
spaces, originally  introduced to this subject in unpublished work of Tataru
and then in  Koch-Tataru \cite{KT}, since they
are ideally suited to time-truncations.  In \S\ref{S:basic}, we discuss
the fundamental dispersive estimates employed in the proofs of the
trilinear estimate and the energy bound.  These include the Strichartz
estimates, local smoothing and maximal function estimates, and
Bourgain's bilinear ``refined Strichartz'' estimates.  In
\S\ref{S:trilinear}, the trilinear estimate is proved along the lines
of Christ-Colliander-Tao \cite{CCT2} and Koch-Tataru \cite{KT}.  In
\S\ref{S:energy}, an energy bound is obtained on a
high-frequency-damped energy functional.  The method here is
essentially an adaptation of the $I$-method of
Colliander-Keel-Staffilani-Takaoka-Tao \cite{CKSTT}.  
Our method does not establish any analogue of this energy
bound for  \emph{differences} of solutions, which is
the reason we cannot obtain a full well-posedness result in $H^s$,
$-\frac18<s<\frac14$.  Finally, in \S\ref{S:proof}, the components are
brought together to give a proof of Theorem \ref{T:main}.

In the conclusion of the introduction we give a heuristic that
explains why, when $s<\frac14$, we expect a piece of the solution at
frequency $N\gg 1$ to propagate according to \emph{linear} dynamics
for at least a time $N^{4s-1} \ll 1$.  Solutions to the linear
equation satisfy the Strichartz estimate (see Lemmas
\ref{L:Strichartz}, \ref{L:smooth_max} below)
\begin{equation}
  \label{E:Strtemp}
  \| D_x^{1/6} e^{-t\partial_x^3}\phi\|_{L_t^6L_x^6} \lesssim \|\phi\|_{L^2} \,.
\end{equation}
Now suppose $u$ is a solution to \eqref{E:mKdV} which is localized at
frequency $N \gg 1$, and suppose $u\approx e^{-t\partial_x^3}\phi$ on
$[0,T]$, with $\| \phi\|_{H^s}\sim 1$.  In the integral equation,
\[
u(t) = e^{-t\partial_x^3}\phi \mp \int_0^t
e^{-(t-t')\partial_x^3}\partial_x u(t')^3 \, dt' \,,
\]
we need to have
\begin{equation}
  \label{E:small}
  \left\| \int_0^t e^{-(t-t')\partial_x^3}\partial_x u(t')^3 \, dt' \, 
\right\|_{L_{[0,T]}^\infty H_x^s} \ll 1 \,.
\end{equation}
We estimate this term as
\[ 
\left\| \int_0^t e^{-(t-t')\partial_x^3}\partial_x u(t')^3 \, dt'
\right\|_{L_{[0,T]}^\infty H_x^s} \leq N^{1+s}
\|u^3\|_{L_{[0,T]}^1L_x^2} \leq T^\frac12 N^{1+s}
\|u\|_{L_{[0,T]}^6L_x^6}^3 \, .
\]
Making the heuristic substitution $u(t) \approx
e^{-t\partial_x^3}\phi$ and applying the Strichartz estimate
\eqref{E:Strtemp},
\[
\| u \|_{L_{[0,T]}^6L_x^6} \approx
\|e^{-t\partial_x^3}\phi\|_{L_{[0,T]}^6L_x^6} \lesssim
N^{-\frac16}\|\phi\|_{L^2} \approx N^{-\frac16-s} \, ,
\]
we see that to achieve \eqref{E:small}, we need $T\lesssim N^{4s-1}$.
Motivated by this, our main function spaces $X^s_M$  defined in the 
next section are constructed by using linear type norms at 
frequency $N$ on the timescale $N^{4s-1}$.

\subsection{Acknowledgments} 
M.C. was supported in part by NSF grant DMS-0901569, J.H. was
supported in part by NSF grant DMS-0901582 and a fellowship from the
Sloan foundation and D.T. was supported in part by NSF grant
DMS-0801261 and by the Miller Foundation.

\section{Function spaces}
\label{S:norms}

We first recall from Koch-Tataru \cite{KT} (see also the careful exposition
in Hadac-Herr-Koch \cite[\S 2]{HHK}) the space-time function spaces
$U^p(I)$ (atomic-space) and $V^p(I)$ (space of functions of bounded
$p$-variation), $1\leq p\leq \infty$.  These are defined on a time
interval $I=[a,b)$, where $-\infty\leq a<b\leq +\infty$ and take values
in $L^2(\mathbb R)$ or any other Hilbert space. Given a
partition $a=t_0<t_1<\cdots<t_K=b$ of $I$ and a sequence $\{ \phi_k
\}_{k=0}^{K-1} \subset L^2_x$ such that $\phi_0=0$ and $\sum_{k=1}^K
\|\phi_{k-1}\|_{L_x^2}^p =1$, the function
\[a(t) = \sum_{k=1}^K \phi_{k-1} \chi_{[t_{k-1},t_k)}(t)\]
is called a $U^p(I)$ atom.  The space $U^p(I)$ is then the collection
of functions $u(t)$ on $I$ of the form
\begin{equation}
  \label{E:atom}
  u(t) = \sum_{\ell=0}^{+\infty} \lambda_\ell a_\ell \,,
\end{equation}
where $a_\ell$ are $U^p(I)$ atoms, with norm
\[
\|u(t)\|_{U^p(I)} = \inf_{\text{representations }
  \eqref{E:atom}}\sum_{\ell=0}^{+\infty} |\lambda_\ell| \,.
\]
It follows that elements $u(t)$ of $U^p(I)$ are right-continuous and
satisfy the boundary conditions
\begin{equation}
  \label{E:u-bdry}
  u(a) = \lim_{t \searrow a} u(t) = 0 \quad \text{and} \quad 
u(b)\defeq \lim_{t\nearrow b} u(t) \text{ exists}\,.
\end{equation}
To define the space $V^p(I)$, we consider functions $v: I\to L_x^2$
such that
\begin{equation}
  \label{E:v-bdry}
  v(a) = \lim_{t\searrow a}v(t) \text{ exists} \quad \text{and}
 \quad v(b)\defeq \lim_{t\nearrow b}v(t)=0 \,,
\end{equation}
and for such functions $v(t)$ define the norm
\[
\|v\|_{V^p(I)} = \sup_{ \{t_k\} } \left(\sum_{k=1}^K
  \|v(t_k)-v(t_{k-1})\|_{L_x^2}^p \right)^{1/p} \,,
\]
where the supremum is taken over partitions $a=t_0<\cdots <t_K=b$.
The fact that the requirement \eqref{E:v-bdry} is preserved in the
limit under the $V^p(I)$ norm follows from \cite[Prop 2.4(i)]{HHK}.

Note that for $I=[a,b)$, $-\infty<a<b<\infty$, we have
\[\| u\|_{U^p(I)} = \| \chi_{I} u \|_{U^p([-\infty,+\infty))}\]
provided $u(a)=0$.  If $u(a)\neq 0$, then the left-side is not defined
(i.e. $u\notin U_p(I)$), while the right-side is defined.  Also,
\[
\|v\|_{V^p(I)} + \|v(a)\|_{L_x^2} = \|\chi_I v
\|_{V^p([-\infty,+\infty))}
\]
provided $v(b)=0$.  If $v(b)\neq 0$, then the left-side is not defined
(i.e. $v\notin V_p(I)$), while the right-side is defined.  Note that a
consequence of \eqref{E:v-trunc} is that for any $v$ with $v(b)=0$, we
have
\begin{equation}
  \label{E:v-trunc}
  \|\chi_I v \|_{V^p([-\infty,+\infty))} \leq 2\|v\|_{V^p(I)} \,.
\end{equation}

\begin{lemma}[$U$-$V$ embeddings]
  \label{L:UVfacts}
  Fix an interval $I=[a,b)$.
  \begin{enumerate}
  \item \label{I:1} If $1\leq p \leq q < \infty$, then
    $\|u\|_{U^q}\leq \|u\|_{U^p}$ and $\|u\|_{V^q} \leq \|u\|_{V^p}$.
  \item \label{I:2} If $1\leq p< \infty$ and $u(b)=0$, then
    $\|u\|_{V^p} \lesssim \|u\|_{U^p}$.
  \item \label{I:3} If $1\leq p<q< \infty$, $u(a)=0$, and $u \in V^p$ is
    right-continuous, then $\|u\|_{U^q} \lesssim \|u\|_{V^p}$.
  \item \label{I:4} Suppose that $1\leq p<q< \infty$, and $T$ is a
    linear operator with the boundedness properties:
\[
\|Tu\|_E \leq C_q \|u\|_{U_A^q}\,, \qquad \|Tu\|_E \leq C_p
\|u\|_{U_A^p} \,, \qquad \text{with } 0< C_p \leq C_q \,,
\]
for some Banach space $E$.  Then
\[
\|Tu\|_E \lesssim \la \ln \frac{C_q}{C_p} \ra \|u\|_{V_A^p} \,,
\]
with implicit constant depending only on the proximity of $q$ and $p$.
\end{enumerate}
\end{lemma}
The first three statements are from Koch-Tataru~\cite{KT}, while the 
last originates in Hadac--Herr--Koch \cite{HHK}.  The
precise references  in \cite{HHK} for all four parts  are: for \eqref{I:1}, see
Prop. 2.2(ii) and Prop. 2.4(iv); for \eqref{I:2}, see Prop. 2.4(iii);
for \eqref{I:3}, see Cor. 2.6; for \eqref{I:4} Prop. 2.17. We
emphasize that in \eqref{I:3}, \eqref{I:4}, we have \emph{strict}
inequality $p<q$.  We also remark that \eqref{I:4} should be thought
of as a quantitative version of \eqref{I:3}.

We now define the space
\[DU^2(I) = \{ \, \partial_t u \, | \, u\in U^2(I) \,\} \,,\]
where the derivative is taken in the sense of distributions.  Given
$f\in DU^2(I)$, a $u\in U^2(I)$ such that $\partial_t u =f$ is in fact
\emph{unique} (recall $u(a)=0$).  Hence we can define
\[\|f\|_{DU^2(I)} = \| u\|_{U^2(I)} \,,\]
which makes $DU^2(I)$ a Banach space.  For example, if $u$ is an atom,
i.e. $u= \sum_{k=1}^K \phi_{k-1} \chi_{[t_{k-1},t_k)}$ with
$a=t_0<\cdots<t_K=b$, $\phi_0=0$ and $\sum_{k=1}^K
\|\phi_{k-1}\|_{L_x^2}^2=1$, then
\[f = \partial_t u= \sum_{k=1}^{K} (\phi_k-\phi_{k-1})\delta_{t_k} \,,\]
(where $\delta_{t_k}$ is the Dirac mass at $t_k$ and we take
$\phi_K\defeq 0$) is an element of $DU^2(I)$ with $\|f\|_{DU^2(I)}=1$.
Note that in this $f$, there is no Dirac mass at position $a$ but
there is one at position $b$ (namely $-\phi_{K-1}\delta_b$). 

\begin{lemma}[$DU$-$V$ duality]
  \label{L:duality}
  We have $(DU^2(I))^*=V^2(I)$ with respect to the usual pairing $\la
  f,v \ra = \int_a^b \la f(t),v(t)\ra_x \, dt= \int_a^b \int_x f \bar
  v \,dx\, dt$.
\end{lemma}
\begin{proof}
  First, we show that if $u\in U^2$ is such that $\partial_t u = f$,
  $u(a)=0$, then $|\la f,v\ra| \leq \|\chi_I u\|_{U^2(I)}
  \|v\|_{V^2(I)}$ for all $v\in V^2(I)$.  Indeed, it suffices to show
  this for $u$ an atom, i.e $u= \sum_{k=1}^K \phi_{k-1}
  \chi_{[t_{k-1},t_k)}$, where $a=t_0<\cdots<t_K=b$ and $\phi_0=0$ and
  $\sum_{k=1}^K \|\phi_k\|_{L_x^2}^2 =1$.  Since $u(a)=0$ and
  $v(b)=0$, we have
  \begin{align*}
    \la f,v\ra
    &= \la \partial_t u, v \ra = - \la u, \partial_t v \ra = -\sum_{k=1}^K \int_a^b \chi_{[t_{k-1},t_k)} \la \phi_{k-1},\partial_t v \ra_x \\
    &= -\sum_{k=1}^K \la \phi_{k-1}, (v(t_k)-v(t_{k-1}))\ra
  \end{align*}
  By Cauchy-Schwarz,
\[| \la f, v \ra | \leq \left(\sum_{k=1}^K \|\phi_{k-1}\|_{L_x^2}^2 \right)^{1/2} \left( \sum_{k=1}^K \|v(t_k)-v(t_{k-1})\|_{L_x^2}^2 \right)^{1/2} \leq \|v\|_{V^2} \,.\]
Next we show that $\sup_{\|f\|_{DU^2(I)}\leq 1} |\la f, v \ra | =
\|v\|_{V^2(I)}$.  Pick a partition $a=t_0<\cdots < t_K=b$ and define
$\phi_0=0$, and for $2\leq k\leq K$ define
\[\phi_{k-1} = \frac{v(t_k)-v(t_{k-1})}{\left(\sum_{j=2}^K \|v(t_j)-v(t_{j-1})\|_{L_x^2}^2\right)^{1/2}}\]
Then, defining $u = \sum_{k=1}^K \phi_{k-1} \chi_{[t_{k-1},t_k)}$ and
$f=\partial_t u$ and arguing as above, $u$ is an atom and
\[\la f,v \ra = \left(\sum_{j=2}^K \|v(t_j)-v(t_{j-1})\|_{L_x^2}^2\right)^{1/2} \,.\]
Taking the supremum over all partitions and using that
$\lim_{t\searrow a} v(t) = v(a)$, we obtain the claim.

Finally, we must show that if $\tilde v\in (DU^2(I))^*$, then there
exists $v\in V^2(I)$ such that $\tilde v(f) = \la f,v\ra$ for all
$f\in DU^2(I)$.  Fix $a<t<b$, and we first define $w(t)$ as follows.
The functional $\phi \mapsto \tilde v( \phi \cdot \delta_t)$ (where
$\delta_t$ is the Dirac mass at $t$) is a bounded linear mapping
$L_x^2 \to \mathbb{C}$.  Hence there exists $w(t)\in L_x^2$ such that
$\la \phi, w(t) \ra_x = \tilde v( \phi \cdot \delta_t)$.  It follows
from \cite[Prop. 2.4(i)]{HHK} that $w(a) \defeq \lim_{t\searrow a}
w(t)$ exists and $w(b)\defeq \lim_{t\nearrow b} w(t)$ exists.  Set
$v(t) = w(t)-w(b)$.  Then if $u$ is an atom in $U^2(I)$ (taking
$\phi_K\defeq 0$ for notational convenience in the summations) and
$f=\partial_t u$,
\begin{align*}
  \la f,v \ra &= \left\la \sum_{k=1}^{K} (\phi_k-\phi_{k-1})\delta_{t_k}, v \right\ra = \sum_{k=1}^K \la (\phi_k-\phi_{k-1}),v(t_k)\ra_x
  = \sum_{k=1}^K  \la (\phi_k-\phi_{k-1}),w(t_k)\ra_x \\
  &= \sum_{k=1}^K  \tilde v( (\phi_k-\phi_{k-1}) \delta_{t_k})
  = \tilde v\left(  \sum_{k=1}^K(\phi_k-\phi_{k-1})\delta_{t_k} \right)
  = \tilde v(f)
\end{align*}

\end{proof}

Now we use the $U^p$ and $V^p$ spaces defined above to construct similar spaces
adapted to the Airy flow. As  base Hilbert spaces
in which functions in $U^p$ and
$V^p$ take values,
 we will use $L^2$, $H^s$, 
as well as a different norm $H^s_M$ on $H^s$ defined by
\[
\|\phi\|_{H^s_M} = \|(|\xi|^2+M)^{\frac{s}2} \hat \phi\|_{L^2}, \qquad M \geq 1 
\]
Finally, for a positive smooth even symbol $a$ satisfying $|a_\xi(\xi)|
\lesssim a(\xi)$ we define the space $H^a$ with norm 
\[
\|\phi\|_{H^a}^2 = \la \phi, a(D) \phi \ra
\]

If the $L^2$ space in the definition of $U^p(I)$, $V^p(I)$ and $DU^2$ spaces
is replaced by another Hilbert space $H \in \{ L^2, H^s, H^s_M,H^a\}$,
we denote the corresponding spaces by $U^2(I;H)$, $V^2(I;H)$, respectively
$DU^2(I;H)$.  Finally, pulling back by the Airy group
$e^{-t\partial_x^3}$ gives the spaces
\[
\|u\|_{U_A^{p}(I;H)} \defeq \| e^{t\partial_x^3} u \|_{U^{p}(I;H)}, \quad 
\|u\|_{V_A^{p}(I;H)} \defeq \| e^{t\partial_x^3} u \|_{V^{p}(I;H)},
\]
\[
\|u\|_{DU_A^{2}(I;H)} \defeq \| e^{t\partial_x^3} u \|_{DU^{2}(I;H)}
\]
The properties in Lemmas~\ref{L:UVfacts},\ref{L:duality}
are easily transferred to this setting.

Consider a dyadic partition of frequencies ($N=2^k$ for some
$k=0,1,\ldots$), $E_N = \{ \xi \, : \, N/2 \leq |\xi| \leq 2N \, \}$,
and let $E_0=[-1,1]$.   Fix consideration to the time interval $[0,1)$.
Consider a smooth Littlewood-Paley partition of unity in frequency 
$1 = \sum P_N$ where each multiplier
$P_N$ is localized to the corresponding set  $E_N$.  For $H$ as above let
\[
\|u\|_{\ell^2L_{[0,1)}^\infty H} \defeq \Big[ \sum_N 
\Big( \|P_N u(t)\|_{L^\infty_{[0,1)} H} \Big)^2 \Big]^{1/2} \, .
\]
Clearly $\|u\|_{L^\infty_{[0,1)} H} \leq
\|u\|_{\ell^2L^\infty_{[0,1)} H}$, but the converse is not true.

To measure the solutions to the mKdV equation we define the spaces 
$X^s_M$ with the norm
\[\|u\|_{X^s_M} \defeq \Big(
\sup_{ |I|=M^{4s-1}}  \|\chi_I P_{\leq M} u\|_{U_A^{2}H^s_M}^2+
 \sum_{N>M}  \sup_{ |I|=N^{4s-1}}  \|\chi_I P_N u\|_{U_A^{2}H^s_M}^2 \Big)^{1/2} \, ,
\]  
where\footnote{Note that here we have written $\|\chi_I P_N
  u\|_{U_A^{s,2}}$ and not $\|P_N u\|_{U_A^{s,2}(I)}$.  Naturally, we
  are not assuming $u$ vanishes at the left endpoint of each of these
  intervals.}  the supremum is taken over all half-open subintervals
$I=[a,b)\subset [0,1)$ of length $N^{1-4s}$.

To measure the nonlinearity in the mKdV equation we define the spaces 
$Y^s_M$ with the norm
\[
\|f\|_{Y^s_M} \defeq \Big( \sup_{ |I|=M^{4s-1}}  \|P_{\leq M} f\|_{DU_A^2 H^s_M}^2+
 \sum_{N > M}  \sup_{ |I|=N^{4s-1}}  \|P_N f\|_{DU_A^{2}(I;H^s_M)}^2 \Big)^{1/2} \, ,
\]  
Similarly we define the space $X^a_M$ and $Y^a_M$.

\section{Basic estimates}
\label{S:basic}

\begin{lemma}
\label{L:equationestimate}
Suppose $\partial_t u + \partial_x^3 u = f$ on $[0,1)$.  Then
\[\| u\|_{X^s_M} \lesssim \|u\|_{\ell^2 L_{[0,1)}^\infty H_M^s} + \|f\|_{Y^{s}_M}\]
\end{lemma}
\begin{proof}
Reduce to the case of a single frequency $N$ by applying $P_N$ to the equation, and then consider a fixed time interval $I=[t_0,t_1)$.  We need to show
\[
\|\chi_I u\|_{U_A^{2} H} \leq \|u(t_0)\|_{H} + \| f\|_{DU_A^{2}(I;H)} \,.
\]
But $\partial_t [ e^{t\partial_x^3}u(t)] = e^{t\partial_x^3}f(t)$, and thus
\[
\|f\|_{DU_A^2(I;H)} = \|e^{t\partial_x^3}f(t)\|_{DU^2(I;H)} = 
\|\chi_I (e^{t\partial_x^3}u(t) -u(a)) \|_{U^2 H} \,.
\]
Hence
\begin{align*}
\| \chi_I u \|_{U_A^2 H} &= \|\chi_I e^{t\partial_x^3}u(t) \|_{U^2 H} \\
&\leq \| \chi_I (e^{t\partial_x^3}u(t) - u(t_0)) \|_{U^2 H} + \|\chi_I u(t_0)\|_{U^2 H}\\
&= \|u(t_0)\|_{H}+ \|f\|_{DU_A^2(I;H)} \,.
\end{align*}
\end{proof}

\begin{lemma}[Bernstein inequality]
\label{L:Bernstein}
For $1\leq p\leq q \leq \infty$,
\[\|P_N f\|_{L^q} \lesssim N^{\frac{1}{q}-\frac{1}{p}} \|f\|_{L^p}\]
\end{lemma}

\subsection{Strichartz, local smoothing, and maximal function
  estimates}

A pair $(p,q)$ of H\"older exponents will be called admissible if
\begin{equation}
  \label{E:admissible}
  \frac2p + \frac1q = \frac12, \qquad 4\leq p\leq \infty, \quad 2\leq q \leq \infty \,.
\end{equation}
In particular, we note that the following pairs $(p,q)$ of indices are
admissible: $(\infty, 2)$, $(6,6)$, $(4,\infty)$.

\begin{lemma}[Strichartz estimates]
  \label{L:Strichartz}
  Let $(p,q)$ satisfy the admissibility condition
  \eqref{E:admissible}.  Then
  \begin{equation}
    \label{E:Str}
    \| D_x^\frac1p e^{-t\partial_x^3}\phi \|_{L_t^pL_x^q} \lesssim \|\phi\|_{L^2}\, .
  \end{equation}
\end{lemma}
In particular, we have, for $N\geq 1$,
\begin{align*}
  & \|P_N e^{-t\partial_x^3} \phi \|_{L_t^\infty L_x^2} \lesssim\|\phi\|_{L^2}, \\
  & \|P_N e^{-t\partial_x^3} \phi \|_{L_t^6 L_x^6} \lesssim N^{-\frac16}\|\phi\|_{L^2}, \\
  & \|P_N e^{-t\partial_x^3}\|_{L_t^4L_x^\infty} \lesssim
  N^{-\frac14}\|\phi\|_{L^2} \,.
\end{align*}
\begin{proof}
  In Kenig-Ponce-Vega \cite{KPV3} Lemma 2.4 / Kenig-Ponce-Vega
  \cite{KPV1} Lemma 3.18(i), the estimate
  \[\| D_x^{\frac14}e^{-t\partial_x^3}\phi\|_{L_t^4L_x^\infty}
  \lesssim \|\phi\|_{L_x^2}\] is proved.  On the other hand, we have
  trivially,
  \[\|e^{-t\partial_x^3}\phi\|_{L_t^\infty L_x^2} = \|\phi\|_{L_x^2}
  \,.\] Now we can apply Stein's theorem on analytic interpolation
  \cite{Stein} to obtain \eqref{E:Str}.
\end{proof}

\begin{lemma}[Local smoothing/maximal function estimates]
  \label{L:smooth_max}
  Let $(p,q)$ satisfy the admissibility condition
  \eqref{E:admissible}.  Then
  \begin{equation}
    \label{E:smooth_max}
    \| D_x^{1-\frac5p} e^{-t\partial_x^3} \phi \|_{L_x^pL_t^q} \lesssim \|\phi\|_{L^2} \, .
  \end{equation}
\end{lemma}
In particular, we note the following estimates, for $N\geq 1$:
\begin{align*}
  & \| P_N e^{-t\partial_x^3} \phi \|_{L_x^\infty L_t^2} \leq cN^{-1}\|\phi\|_{L^2}, \\
  & \| P_N e^{-t\partial_x^3} \phi \|_{L_x^6 L_t^6} \leq cN^{-\frac16}\|\phi\|_{L^2}, \\
  & \| P_N e^{-t\partial_x^3} \phi \|_{L_x^4 L_t^\infty} \leq
  cN^{\frac14}\|\phi\|_{L^2} \, .
\end{align*}
\begin{proof}
  The local smoothing estimate (Kenig-Ponce-Vega \cite{KPV1}, Theorem
  3.5(i)) is
  \[\|\partial_x e^{-t\partial_x^3}\phi\|_{L_x^\infty L_t^2} \lesssim
  \|\phi\|_{L^2} \,.\] It is basically reducible to Plancherel in $t$.
  On the other hand, we have the maximal function estimate
  (Kenig-Ponce-Vega \cite{KPV1}, Theorem 3.7(i) on p. 556)
  \[\| D_x^{-\frac14} e^{-t\partial_x^3} \phi\|_{L_x^4L_t^\infty}
  \lesssim \|\phi\|_{L^2} \,.\] It is proved by reducing by duality
  and a $TT^*$ argument to an estimate that is proved by the theorem
  on fractional integration and a pointwise Airy function estimate.
  We now apply Stein's theorem on analytic interpolation \cite{Stein}
  to obtain \eqref{E:smooth_max}.
\end{proof}

The next two corollaries are consequences of these estimates, and
relate the Strichartz space-time norms to the Airy-atomic norm
$U_A^{s,2}$ norm of \emph{any} function $u(x,t)$ (not necessarily a
solution to the linear Airy equation).

\begin{corollary}
  \label{C:X_Str}
  If $I=[a,b)$ is any interval, and $u=u(x,t)$ any function, then for
  $(p,q)$ satisfying the admissibility condition \eqref{E:admissible},
  we have, for $N\geq 1$,
  \begin{equation}
    \label{E:X_Str}
    \| P_N u \|_{L_I^pL_x^q} \lesssim N^{-\frac1p}\|\chi_I u\|_{U_A^{p} L^2} \, ,
  \end{equation}
  and we have the dual relation for $p>2$
  \begin{equation}
    \label{E:X_Str_dual}
    \|P_N u \|_{DU_A^{2}(I;L^2)} \lesssim N^{-\frac1p}\|u\|_{L_{I}^{p'}L_x^{q'}}\, ,
  \end{equation}
  where $(p',q')$ denotes the H\"older dual pair.
\end{corollary}
\begin{proof}
  To prove \eqref{E:X_Str}, it suffices to assume
  $I=[-\infty,+\infty)$, since $\chi_I$ can be inserted.  It also
  suffices to consider a $U_A^p$-atom
  \begin{equation}
    \label{E:u-atom}
    u(t,x) = \sum_{k=1}^K \chi_{[t_{k-1},t_k)}(t)e^{-t\partial_x^3}\phi_{k-1}(x), \quad \sum_{k=1}^K \|\phi_{k-1}\|_{L_x^2}^p = 1 \,, \quad \phi_0=0 \,,
  \end{equation}
  and prove that
  \begin{equation}
    \label{E:atomic-Str}
    \| P_N u \|_{L_t^p L_x^q} \lesssim N^{-\frac1p} \,.
  \end{equation}
  But \eqref{E:atomic-Str} follows directly from \eqref{E:Str}, as
  follows:
  \begin{align*}
    \| P_N u \|_{L_t^pL_x^q}^p &=  \sum_{k=1}^K \|\chi_{[t_{k-1},t_k)}(t) P_N e^{-t\partial_x^3} \phi_{k-1}\|_{L_t^p L_x^q}^p \\
    &\lesssim N^{-1} \sum_{k=1}^K \|\phi_{k-1}\|_{L_x^2}^p
    = N^{-1} \,.
  \end{align*}
  To prove \eqref{E:X_Str_dual}, note that since $(DU^2(I;L^2))^*=V^2(I;L^2)$,
  we have
  \[
\| P_N u\|_{DU_A^{2}(I;L^2)} = \sup_{\|v\|_{V_A^2(I;L^2)} \leq 1} \int_I
  \int_x P_N u \, \bar v \, dx\,dt \,.
\]
 But
  \[
| \la P_Nu,v \ra | \leq \|u\|_{L_{I'}^{p'} L_x^{q'}} \|P_N v
  \|_{L_I^pL_x^q} \,,
\]
and by \eqref{E:X_Str} and Lemma \ref{L:UVfacts}\eqref{I:3} (applied
on the interval $[-\infty,+\infty)$), we have, for $p>2$,
  \[
\|P_N v \|_{L_I^pL_x^q} \lesssim \|\chi_I v\|_{U_A^p L^2} \lesssim
  \|\chi_Iv\|_{V_A^2 L^2} \,.
\] 
Apply \eqref{E:v-trunc} ($\|\chi_I
  v\|_{V_A^2 L^2} \leq 2\|v\|_{V_A^2(I;L^2)}$) to complete the proof.
\end{proof}

\begin{corollary}
  \label{C:X_sm}
  If $(p,q)$ is admissible according to \eqref{E:admissible} and
  $p,q\geq r$, then
  \begin{equation}
    \label{E:X_sm}
    \|P_N u\|_{L_x^pL_I^q}  \lesssim N^{\frac5p-1} \|\chi_I u\|_{U_A^{r} L^2} \,.
  \end{equation}
  for any interval $I=[a,b)$.  We also have the dual relation for
  $q>2$,
  \begin{equation}
    \label{E:X_sm_dual}
    \|P_N u\|_{DU_A^{2}(I;L^2)} \lesssim N^{\frac5p-1} \|u\|_{L_x^{p'}L_{I}^{q'}} \,,
  \end{equation}
  where $(p',q')$ is the H\"older dual pair.
\end{corollary}
\begin{proof}
  As we argued in the proof of Cor. \ref{C:X_Str}, it suffices to
  prove \eqref{E:X_sm} for $u$ an atom of the form \eqref{E:u-atom}
  (with $p$ replaced by $q$) on $I=[-\infty,+\infty)$.  For such $u$
we write
\[
u = \sum u_k,  \qquad u_k = \chi_{[t_{k-1},t_k)}(t)P_N 
e^{-t\partial_x^3}\phi_{k-1}
\]
Applying \eqref{E:smooth_max} for each $u_k$, it remains to show that 
\[
\| u \|_{L^p_x L^q_t}^r \lesssim \sum_k \| u_k \|_{L^p_x L^q_t}^r
\]
or equivalently
\[
\| |u|^r \|_{L^\frac{p}r_x L^\frac{q}r_t} \lesssim \sum_k \| |u_k|^r
\|_{L^\frac{p}r_x L^\frac{q}r_t}
\]
But $u_k$ have disjoint supports therefore $|u|^r = \sum |u_k|^r$
and the last relation follows by the triangle inequality.

  For \eqref{E:X_sm_dual}, we note that since $(DU_A(I;L^2))^*=V_A^2(I;L^2)$
  \[
  \|P_N u\|_{DU_A^{2}(I;L^2)} = \sup_{\|v\|_{V_A^2(I;L^2)}=1} \left| \int_I
    \int_x P_N u \; \bar v \,dx\,dt \right| \,.\] But by H\"older,
  \[
  \left|\int_I \int_x P_N u \; \bar v \,dx\,dt \right| \leq
  \|u\|_{L_I^{p'}L_x^{q'}} \|P_N v\|_{L_I^pL_x^q} \,,
  \]
  and by \eqref{E:X_sm} and for $q>2$, we have
  \[\|P_N v\|_{L_I^pL_x^q} \leq \|\chi_I v\|_{U_A^q L^2} \leq \|\chi_I
  v\|_{V_A^2 L^2} \,.\] Finally apply \eqref{E:v-trunc} to obtain the
  bound by $\|v\|_{V_A^2(I;L^2)}$.
\end{proof}

\subsection{Bilinear estimate}

\begin{lemma}[Bilinear estimate]
  \label{L:bilinear}
  Suppose $E_1, E_2 \subset \mathbb{R}$ and $M_1, M_2>0$ are dyadic
  values (no restriction to $\geq 1$) such that
  \[\forall \; \xi_1\in E_1 \text{ and }\xi_2\in E_2, \qquad
  |\xi_1+\xi_2|\sim M_1 \text{ and } |\xi_1-\xi_2|\sim M_2 \, .\] Let
  $P_j$ be the $x$-frequency projection operators defined as
  $\widehat{P_j f}(\xi) = \chi_{E_j}(\xi)\hat f(\xi)$ for a function
  $f=f(x)$.  Then,
  \begin{equation}
    \label{E:bi}
    \| P_1 e^{-t\partial_x^3} \phi \; P_2 e^{-t\partial_x^3}\psi \|_{L_t^2L_x^2} \lesssim (M_1M_2)^{-\frac12} \|P_1\phi\|_{L^2}\|P_2\psi\|_{L^2}\,.
  \end{equation}
\end{lemma}
\begin{proof}
  \[ [ P_1 e^{-t\partial_x^3} \phi \; P_2
  e^{-t\partial_x^3}\psi]\widehat{\phantom{f}}(\xi,t) =
  \int_{\substack{\xi_1\in E_1 \\ \xi_2\in E_2 \\ \xi=\xi_1+\xi_2}}
  e^{it\xi_1^3}\hat\phi(\xi_1) e^{it\xi_2^3}\hat\psi(\xi_2) \] and
  thus
  \begin{align*}
    [P_1 e^{-t\partial_x^3} \phi \; P_2 e^{-t\partial_x^3}\psi]\widehat{\phantom{f}}(\xi,\tau) &= \int_{\substack{\xi_1\in E_1 \\ \xi_2\in E_2 \\ \xi=\xi_1+\xi_2}} \delta(\tau-\xi_1^3-\xi_2^3)\hat\phi(\xi_1) \hat\psi(\xi_2)\\
    &= \frac{\chi_{E_1}(\xi_1)\chi_{E_2}(\xi_2)
      \hat\phi(\xi_1)\hat\psi(\xi_2)}{3(\xi_1^2-\xi_2^2)}
  \end{align*}
  where, in the last line, $(\xi_1,\xi_2)$ is the solution to
  \[\tau=\xi_1^3+\xi_2^3, \qquad \xi=\xi_1+\xi_2\, .\]
  [In fact, there could be 0, 1, or 2 solutions $(\xi_1,\xi_2)$
  depending upon the particular $(\xi,\tau)$; a proper argument would
  exhibit these regions separately, etc.]  The Jacobian for the change
  of variable $(\xi,\tau) \mapsto (\xi_1,\xi_2)$ is
  \[d\tau d\xi = 3|\xi_1^2-\xi_2^2| d\xi_1d\xi_2 \, .\] The result
  then follows from Plancherel's theorem and this change of variable.
\end{proof}

\begin{corollary}
  \label{C:bilinear}
  Under the hypothesis of Lemma \ref{L:bilinear}, if $u=u(x,t)$,
  $v=v(x,t)$ are any functions, then\footnote{Note the use of the
    truncation functions $\chi_I$ and then evaluation in
    $U_A^{0,2}([-\infty,+\infty))$ or $V_A^{0,2}([-\infty,+\infty))$
    on the right-side.  We are not using the norms $U_A^{0,2}(I)$ or
    $V_A^{0,2}(I)$, since they require vanishing at the left and right
    endpoints of $I$, respectively.  We do not want to impose such a
    condition for finite-length intervals $I$.}
  \begin{equation}
    \label{E:X_bi}
    \| P_1 u \; P_2 v \|_{L_I^2L_x^2} \lesssim (M_1M_2)^{-\frac12} \|\chi_I P_1u\|_{U_A^{2} L^2}\|\chi_I P_2v\|_{U_A^{2} L^2} \,
  \end{equation}
  \begin{equation}
    \label{E:X_bi_V}
    \| P_1 u \; P_2 v \|_{L_I^2L_x^2} \lesssim (M_1M_2)^{-\frac12} 
\left\la \ln \frac{M_1}{M_2} \right\ra^2 \|\chi_I P_1u\|_{V_A^{2} L^2}
\|\chi_I P_2v\|_{V_A^{2} L^2} \,
  \end{equation}
\end{corollary}
\begin{proof}
  It clearly suffices to prove the estimates for
  $I=[-\infty,+\infty)$, since we can insert $\chi_I$ cutoffs on $u$
  and $v$.  We begin noting that if we fix $u=e^{-t\partial_x^3}\psi$,
  and $v$ a $U_A^2$ atom, i.e.
  \[v(x,t)=\sum_{k=1}^K
  \chi_{[t_{k-1},t_k)}(t)e^{-t\partial_x^3}\phi_{k-1} \,, \quad
  \phi_0=0 \,, \quad \sum_{k=1}^K \|\phi_{k-1}\|_{L_x^2}^2=1 \,,\]
  then it follows from Lemma \ref{L:bilinear} that
  \begin{equation}
    \| P_1 u \; P_2 v \|_{L_I^2L_x^2} \lesssim (M_1M_2)^{-\frac12} \|\psi\|_{L^2}  \,.
  \end{equation}
  By linearity in $u$, we obtain the estimate \eqref{E:X_bi} when both
  $u$ and $v$ are $U_A^2$ atoms.  The general case of \eqref{E:X_bi}
  follows by linearity and density.  The estimate \eqref{E:X_bi_V}
  follows from \eqref{E:X_bi} by the argument in \cite[Cor. 2.18]{HHK}
  which appeals to their Prop. 2.17 (our Lemma
  \ref{L:UVfacts}\eqref{I:4}).
\end{proof}

\section{Trilinear estimate}
\label{S:trilinear}

\begin{proposition}[Trilinear estimate]
  \label{P:trilinear}
For all  $-\frac17 < s < \frac14$ and $M \geq 1$ we have
\[
\| \partial_x (u_1u_2u_3) \|_{Y^s_M} \lesssim \| u_1 \|_{X^s_M} \| u_2 \|_{X^s_M} 
\| u_3 \|_{X^s_M} \,.
\] 
\end{proposition}
\begin{proof}
  We insert frequency projections $P_{N_j}$, $P_N$ where $N,N_j \geq M$.
  Denoting the truncated functions by $u_{N_j} = P_{N_j} u_j$ for $N_j > M$ while $ u_{N_j} =
  P_{<M} u_j$ for $N_j = M$, we reduce matters to
  proving, for an interval $|J| = N^{4s-1}$ with $N > M$, a bound of the type
  \begin{equation}
    \label{E:tri}
    \| P_N  \partial_x (u_{N_1} \; u_{N_2} \; u_{N_3}) \|_{DU_A^{2}(J;H^s_M)} \leq \alpha (N,N_1,N_2,N_3) \prod_{j=1}^3 \sup_{|I_j|=N_j^{4s-1}}\|\chi_{I_j} u_{N_j} \|_{U_A^{2} H^s_M}
  \end{equation}
  as well as the similar bound with $ P_N$ replaced by $P_{<M}$. This
  can be rewritten as
  \begin{align*}
    \indentalign \| P_N \partial_x (u_{N_1} \; u_{N_2} \; u_{N_3}) \|_{DU_A^2(J;L^2)} 
\leq \alpha (N,N_1,N_2,N_3) \frac{N_1^sN_2^sN_3^s}{N^s} \prod_{j=1}^3
    \sup_{|I_j|=N_j^{4s-1}}\|\chi_{I_j} u_{N_j} \|_{U_A^2L^2}
  \end{align*}
  Here $\alpha $ should have certain summability properties.  As a
  general rule, we need at least that $|\alpha (N,N_1,N_2,N_3)|
  \lesssim 1$, and in some cases, need a slight power decay in $N$
  and/or $N_j$ to insure the summation with respect to all indices.
\bigskip

\noindent \textit{Case 1}.  $N_1, N_2, N_3 \lesssim N$.  We can assume
that $N_1 \leq N_2 \leq N_3 \sim N$.  In this case, all $I_j$ have
length $\geq |J|$ and can be neglected. We distribute the derivative,
which in the worst case applies to $u_{N_3}$. By \eqref{E:X_Str_dual}
and \eqref{E:X_sm},
  \begin{align*}
    \| P_N (u_{N_1} \; u_{N_2} \;  \partial_x u_{N_3}) \|_{DU_A^2(J;L^2)}
    &\lesssim \|u_{N_1} \; u_{N_2} \;  \partial_x u_{N_3}\|_{L_J^1L_x^2} \\
    &\lesssim |J|^\frac12 \|u_{N_1} \; u_{N_2} \;  \partial_x u_{N_3} \|_{L_J^2L_x^2}\\
    &\lesssim N^{2s-\frac12} \| u_{N_1} \|_{L_x^4 L_J^\infty} \|  u_{N_2} \|_{L_x^4L_J^\infty}  \| \partial_x u_{N_3}\|_{L_x^\infty L_J^2} \\
    &\lesssim N^{2s-\frac12} N_1^\frac14\|\chi_{J} u_{N_1}
    \|_{U_A^2 L^2} N_2^\frac14\|\chi_{J} u_{N_2} \|_{U_A^2 L^2}
    \| \chi_{J} u_{N_3} \|_{U_A^2 L^2} 
  \end{align*}
  Thus we have \eqref{E:tri} with $\alpha =
  N^{2s-\frac12}N_1^{\frac14-s}N_2^{\frac14-s}$, which suffices 
 for all $s$.

\bigskip

  \noindent \textit{Case 2}.\! $N_1 \lesssim\! N \ll\! N_2\sim \!N_3$.  
The $u_2, u_3$ terms need to be evaluated in norms restricted to
  intervals $I$ of size $|I|= N_3^{4s-1}$.  We divide $J$ into $|J|/|I| =
  (N_3/N)^{1-4s} \gg 1$ intervals of size $|I| = N_3^{4s-1}$.  For
  $u\in V_A^2(J;L^2)$ we estimate by duality
  (Lemma \ref{L:duality})
  \begin{align*}
     \Big| \int_{J} \int_x u_{N_1}  u_{N_2} \, u_{N_3} \, u_N \, dx\, dt \Big| 
    &\leq \Big( \frac{N_3}{N} \Big)^{1-4s} \sup_{\substack{I\subset J\\ |I|=N_3^{4s-1}}} \Big| \int_{ I} \int_x u_{N_1} \, u_{N_2} \, u_{N_3} \, u_N  \, dx\, dt \Big| \\
    &\leq \Big( \frac{N_3}{N} \Big)^{1-4s} \sup_{\substack{I\subset
        J\\ |I|=N_3^{4s-1}}} \|u_{N_1} \, u_{N_2}
    \|_{L_I^2L_x^2} \|u_{N}  \, u_{N_3} \|_{L_I^2L_x^2} \,.
  \end{align*}
  Using the bilinear estimate \eqref{E:X_bi},\eqref{E:X_bi_V}
 we bound the above by
\[
\Big( \frac{N_3}{N} \Big)^{1-4s} N_3^{-2} \left\la \ln \frac{N_3}{N} \right\ra^2 
\sup_{\substack{I\subset J\\ |I|=N_3^{4s-1}}} 
\|\chi_I u_{N_1}  \|_{U_A^2 L^2} \|\chi_I u_{N_2} \|_{U_A^2 L^2} 
\|\chi_I u_{N_3} \|_{U_A^2 L^2} \|\chi_I  u_N \|_{V_A^2 L^2} \,.
\]
Finally, we apply \eqref{E:v-trunc} ($\|\chi_IP_N u\|_{V_A^2} \leq
2\|P_N u\|_{V_A^2(J)}$).  Adding a factor of $N$ to account for the
derivative in \eqref{E:tri} we obtain
\[
\alpha =
N_3^{-1-6s}N^{5s}N_1^{-s} \left\la \ln \frac{N_3}{N} \right\ra^2
\]
 so this case is handled if $s\geq -\frac16$.

\bigskip 

\noindent \textit{Case 3}.  $N \ll N_1 \leq N_2=N_3$.  We can assume
that $\xi_1$, $\xi_2$ have the same sign and that $\xi_3$ has the
opposite sign.  [Indeed, if $N_1\ll N_2$, then this is achieved by
permuting $N_2$ and $N_3$ if necessary, and if $N_1\sim N_2\sim N_3$,
then this can be arranged by permuting the indices.]  Note that then
obviously we have $|\xi_1-\xi_3|\sim N_3$, but also since $N \ll
N_2\sim N_3$, we have $|\xi_1+\xi_3|=|\xi+\xi_2|\sim N_3$ and
$|\xi-\xi_2|\sim N_3$.

We again argue by duality (Lemma \ref{L:duality}) and divide into
subintervals of size $|I|=N_3^{4s-1}$.  For $v\in V_A^2(J;L^2)$,
\begin{align*}
  \Big| \int_{t\in J} \int_x u_{N_1} \, u_{N_2} \, u_{N_3} \,  u_N \, dx\, dt \Big| 
  &\leq \Big( \frac{N_3}{N} \Big)^{1-4s}
 \sup_{\substack{I\subset J\\ |I|=N_3^{4s-1}}} 
\Big| \int_{t\in I} \int_x u_{N_1} \, u_{N_2} \, u_{N_3} \, u_N u \, dx\, dt \Big| \\
  &\leq \Big( \frac{N_3}{N} \Big)^{1-4s} \sup_{\substack{I\subset J\\
      |I|=N_3^{4s-1}}} \|u_{N} \, u_{N_2} \|_{L_I^2L_x^2}
  \|u_{N_1} \, u_{N_3} \|_{L_I^2L_x^2} \,.
\end{align*}
We then apply the bilinear estimate \eqref{E:X_bi}, \eqref{E:X_bi_V}
to bound the above by
\[
\leq
 \Big( \frac{N_3}{N} \Big)^{1-4s} N_3^{-2} \left\la \ln \frac{N_3}{N} \right\ra^2 \sup_{\substack{I\subset J\\ |I|=N_3^{4s-1}}}  \| \chi_I u_{N_1} \|_{U_A^2L^2}
 \|\chi_I u_{N_2}  \|_{U_A^2L^2} \| \chi_I u_{N_3} \|_{U_A^2 L^2} 
\| \chi_I u_N \|_{V_A^2 L^2} 
\]
Finally, we apply \eqref{E:v-trunc}.  Thus we have $\alpha =
N_3^{-1-6s}N^{5s}N_1^{-s}$, which is satisfactory if $s>-\frac17$.

\end{proof}

\section{Energy bound}
\label{S:energy}

For expositional convenience, in this section, we will assume that we
are in the more difficult case $s\leq 0$.  We study the almost
conservation of the $H^s$ norm using a variant of the $I$-method of
Colliander-Keel-Staffilani-Takaoka-Tao \cite{CKSTT}.  The main result
of this section is as follows:

\begin{proposition}[Energy bound]
  \label{P:energybound}
  For all $-\frac18 \leq s \leq  0$,  $M > 0$ and 
$u$ solving \eqref{E:mKdV}  we have the following bound in the 
time interval $[0,1]$:
\begin{equation}\label{l2-energy} 
 \|u\|_{\ell^2 L^\infty H^s_M}^2  \leq c( \|u\|_{\ell^2 L^\infty H^s_M}^4+  
  \|u\|_{X^s_M}^6)
\end{equation}
\end{proposition}

Due to the $l^2$ dyadic summation on the left we cannot simply 
obtain a uniform in time bound for the $H^s$ norm of $u$.
Instead for small $\epsilon > 0$ we introduce a class $S_M$
of real smooth positive even symbols $a(\xi)$ which have the 
following  properties:
\begin{itemize}
\item[(i)]  $a(\xi)$ is constant for $|\xi| \leq M$.
\item[(ii)] Regularity:
\begin{equation} \label{adiff}
|\partial_\xi^\alpha a(\xi)| \leq c_\alpha a(\xi) \la \xi \ra^{-\alpha}
\end{equation}
\item[(iii)] Decay properties 
\begin{equation}\label{asize}
- \frac12  \leq \frac{d \log a(\xi)}{d \log (1+ \xi^2)} \leq 0
\end{equation}
\end{itemize}
The latter property implies that  $a(\xi)$  is nonincreasing
but decays  no faster than\footnote{In effect decay rates 
up to $|\xi|^{-1}$ are still acceptable, but not needed here.}
 $|\xi|^{-\frac12}$.
For $a \in S_M$ we will prove the uniform bound 
\begin{equation}\label{energy} 
 \|u\|_{L^\infty H^a}^2  \leq \|u(0)\|_{H^a}^2  + c( 
\|u\|_{L^\infty H^s_M}^2 \|u\|_{ L^\infty H^a}^2 +  
  \|u\|_{X^s_M}^4 \|u\|_{X^a_M}^2)
\end{equation}
which implies the desired bound \eqref{l2-energy}. To see this, for
each dyadic $N \geq M$ we consider a symbol 
$a_{N} \in S_M$ such that
\[
a_{N}(\xi) \defeq \begin{cases} N^{2s} & \text{if } |\xi|\leq N \\
  N^{\frac12+2s} |\xi|^{-\frac12} & \text{if }|\xi|\geq 2N \end{cases} \, .
\]  
Then \eqref{l2-energy} follows from \eqref{energy} applied to $a_N$
due to  the obvious relations
\[
\| u\|_{\ell^2 L^\infty H^s_M}^2 \approx \sum_{N \geq M} 
\| u\|_{L^\infty H^{a_N}}^2 ,
\]
\[ 
 \|u\|_{X^s_M}^2
\approx  \sum_{N \geq M}
\| u\|_{X^a_N}^2
\]
It remains to prove the bound \eqref{energy}. We define the 
energy functional 
\[
E_0(u)  \defeq  \la A(D)u, u \ra =   \| u\|_{H^a}^2
\]
and compute its derivative  along the flow.
Since $a(\xi)$ is even and $u$ is real, $A(D)u$ is
real.  Also, $A(D)$ is self-adjoint since $a(\xi)$ is
real.  Thus, substituting \eqref{E:mKdV},
\[
\frac{d}{dt} E_0(u) = 
R_4(u) \defeq \pm 2 \la A(D)\partial_x u, u^3
\ra \,.
\]
Using the fact that $u$ is a real valued function,
which implies that $\overline{\hat u(-\xi)} = \hat u
(\xi)$, we write $R_4$ as a multilinear operator in Fourier
space:
\[
R_4(u) = 
\pm 2 \int_{P_4} i\xi_1 a(\xi_1) \; \hat u(\xi_1) \hat
u(\xi_2) \hat u(\xi_3) \hat u(\xi_4) \, d\sigma \,,
\]
where $P_4 = \{ \, (\xi_1,\xi_2,\xi_3,\xi_4)\in\mathbb{R}^4 \; | \;
\xi_1+\xi_2+\xi_3+\xi_4=0 \, \}$.  This expression for $R_4$ can be
symmetrized as
\[
R_4(u) = \pm \frac12 \int_{P_4} i ( \xi_1 a(\xi_1)+ \xi_2 a(\xi_2)+
\xi_3 a(\xi_3)+ \xi_4 a(\xi_4)) \hat u(\xi_1) \hat u(\xi_2) \hat
u(\xi_3) \hat u(\xi_4) \, d\sigma \,.
\]
We seek to cancel this term by perturbing the energy to $E_0+E_1$,
where $E_1$ has the form
\[
E_1(u) = \int_{P_4} b_4(\xi_1,\xi_2,\xi_3,\xi_4) \hat u(\xi_1) \hat
u(\xi_2) \hat u(\xi_3)\hat u(\xi_4) \, d\sigma \,.
\]
To determine the proper choice for $b_4$, we compute
\begin{align*}
  \frac{d}{dt} E_1(u) = &\int_{P_4} b_4(\xi_1,\xi_2,\xi_3,\xi_4) 
i(\xi_1^3+\xi_2^3+\xi_3^3+\xi_4^3) \; 
\hat u(\xi_1) \hat u(\xi_2) \hat u(\xi_3) \hat u(\xi_4) \, d\sigma 
+ R_6(u) \,,
\end{align*}
where $R_6(u)$ has the form (if we for convenience go ahead and 
assume that $b_4$ is symmetric under exchange of any pair 
from $\xi_1$, $\xi_2$, $\xi_3$ and $\xi_4$)
\[
R_6(u) = \mp \frac14 \int_{P_4} b_4(\xi_1,\xi_2,\xi_3,\xi) i\xi
\widehat{u^3}(\xi) \hat u(\xi_1) \hat u(\xi_2) \hat u(\xi_3) \,
d\sigma \,.
\]
Now we see that the proper choice of $b_4$ to cancel the term $R_4$ is
\[
b_4(\xi_1,\xi_2,\xi_3,\xi_4) = \pm \frac12 \frac{\xi_1a(\xi_1)+
  \xi_2a(\xi_2) + \xi_3a(\xi_3) +
  \xi_4a(\xi_4)}{\xi_1^3+\xi_2^3+\xi_3^3+\xi_4^3}
\]
In conclusion, we have
\[
\frac{d}{dt}(E_0+E_1)(u) = R_6(u) \,.
\]
Hence in order to prove  \eqref{energy} we need to establish
the following two bounds:
\begin{equation} \label{E1bd}
E_1(u) \lesssim \|u\|_{H^s_M}^2 \|u\|_{H^a}^2 
\end{equation}
respectively
\begin{equation} \label{R6bd}
\int_{0}^1 R_6(u(t)) dt \lesssim   
\|u\|_{X^s_M}^4 \|u\|_{X^a}^2
\end{equation}
In order to do this we need to study the size and regularity 
of $b_4$.

\begin{lemma} 
Let $a \in S_M$. Then there exists a  symbol $b_4$
in $\mathbb R^4$ so that
\begin{equation}
\xi_1a(\xi_1)+ \xi_2a(\xi_2) + \xi_3a(\xi_3) + \xi_4a(\xi_4)
= b_4 (\xi_1^3+\xi_2^3+\xi_3^3+\xi_4^3)  \qquad \text{ on } P_4
\label{b4-eq}\end{equation}
with the following size and regularity
in dyadic regions $\{|\xi_j| \sim N_j > M \}$ respectively 
$\{\xi_j \lesssim N_j = M\}$:
\begin{equation}\label{b4bd}
| \partial_{\xi_1}^{\alpha_1}   \partial_{\xi_2}^{\alpha_2}  
\partial_{\xi_3}^{\alpha_3} \partial_{\xi_4}^{\alpha_4}b_4
(\xi_1,\xi_2,\xi_3,\xi_4) | \leq c_{\alpha}    b_4 (N_1,N_2,N_3,N_4) 
 N_1^{-\alpha_1}
  N_2^{-\alpha_2}   N_3^{-\alpha_3} N_4^{-\alpha_4},
\end{equation}
where
\[
 b_4 (N_1,N_2,N_3,N_4) = a(N_2) N_4^{-2} \qquad \text{when }
N_1 \leq N_2 \leq N_3 \sim N_4
\]
\label{L:b4bd}
\end{lemma}
\begin{proof}
On $P_4$, we have the factorization
\[
\xi_1^3+\xi_2^3+\xi_3^3+\xi_4^3 = (\xi_1+\xi_2)(\xi_1+\xi_3)(\xi_2+\xi_3) \,.
\]
Let $N_j \geq M$ denote the dyadic zone of $|\xi_j|$ (as before the
$M$ dyadic zone includes all frequencies below $M$).  On $P_4$ we
necessarily have $N_3\sim N_4$.  If all $|\xi_j| \leq M$, then the left hand
side of \eqref{b4-eq} is zero since
$a(\xi)=\text{const}$ for $|\xi|\leq M$, we have that .  Therefore, we take $b_4 = 0 $ there
and assume $N_4\geq M$ in the remainder of the proof.  We consider
several cases.
\bigskip

  \noindent \textit{Case 1}.  $N_1\ll  N_2 \leq  N_3\sim N_4$.  Then we 
define 
\[
b_4(\xi) = - \frac{\xi_1a(\xi_1)+ \xi_2a(\xi_2)} 
{ (\xi_1+\xi_2)(\xi_1+\xi_3)(\xi_1+\xi_4)} -\frac{1}{(\xi_1+\xi_3)(\xi_1+\xi_4)}
 \frac{ \xi_3a(\xi_3) + \xi_4a(\xi_4)}
{ \xi_3+\xi_4} 
\]
Since $|\xi_1+\xi_2| \sim N_2$ and $|\xi_1+\xi_3|, |\xi_1+\xi_4| \sim N_4$,
the conclusion easily follows by taking advantage of the cancellation 
in the last fraction when $\xi_3+\xi_4=0$.

\bigskip

  \noindent \textit{Case 2}.  $N_1 \sim  N_2 \ll  N_3\sim N_4$. 
Then we have 
$|\xi_1+\xi_3|, |\xi_2+\xi_3|, |\xi_1+\xi_2+\xi_3|  \sim N_4$.
Hence we  define 
\begin{equation}\label{b4diff}
  b_4(\xi) =  \frac{\xi_1a(\xi_1)+ \xi_2a(\xi_2) + \xi_3a(\xi_3) - 
(\xi_1+\xi_2+\xi_3)  a(-\xi_1-\xi_2- \xi_3)}
  { (\xi_1+\xi_2)(\xi_1+\xi_3)(\xi_2+\xi_3)}
\end{equation}
and the only difficulty comes from the division by $\xi_1+\xi_2$.
We rewrite $b_4$ as 
\[
b_4(\xi) =  \frac{1}{(\xi_1+\xi_3)(\xi_2+\xi_3)}( g(\xi_1,\xi_2) -
g(\xi_3,-\xi_1-\xi_2-\xi_3))
\]
where the function $g$ is defined by
\[
g(\xi,\eta) = \frac{\xi a(\xi) +\eta a(\eta)}{\xi+\eta}.
\]
Since $a$ is even and satisfies \eqref{adiff}, it follows that 
$g$ is smooth on the dyadic scale and has size $\lesssim a(N)$ 
when $|\xi| \sim |\eta| \sim N$. The conclusion again follows.

\bigskip
\noindent\textit{Case 3}.  $N_1\sim N_2\sim N_3\sim N_4 \sim N$. 
Using a partition of unit on the $N$ scale and permuting the indices
we can assume that we localized the problem to a region where
$|\xi_2+\xi_3|,|\xi_1+ \xi_2+\xi_3| \sim N$.  Then we define $b_4$
using again \eqref{b4diff}, and rewrite it in the form
\[
b_4(\xi) =  \frac{1}{\xi_2+\xi_3}  \frac{g(\xi_1,\xi_2) -
g(\xi_3,-\xi_1-\xi_2-\xi_3))}{(\xi_1+\xi_3)}
\]
Now the first factor is elliptic, and in the second factor 
the numerator vanishes on $\{\xi_1+\xi_3 = 0\}$ therefore
we have again a smooth division on the dyadic scale.
\end{proof}

The next result implies the bound \eqref{E1bd}:
\begin{corollary}
  \label{C:E-bds}
Let $a \in S_M$ and $b_4$ as in Lemma~\ref{L:b4bd}.
Then
\begin{equation}
|E_1(u)| \lesssim   \|u\|_{H^a}^2 \|u\|_{H^{-\frac12}_M}^2
\end{equation}
\end{corollary}
\begin{proof}
Given the expression of $b_4$, it suffices to prove this 
when $\hat u$ is positive and $b_4$ is estimated pointwise by 
\eqref{b4bd}. Using again the notation $u_N = P_N u$ 
for $N > M$ and $u_M = P_{\leq M} u$,
by Bernstein's inequality  we have
\[
\begin{split}
  |E_1(u)| \lesssim & \ \sum_{M \leq N_1 \leq N_2 \leq N_3 \sim N_4}
  \frac{a(N_2)}{N_4^{2}} \| u_{N_1} u_{N_2} u_{N_3} u_{N_4} \|_{L^1}
  \\
  \lesssim & \ \sum_{M \leq N_1 \leq N_2 \leq N_3 \sim N_4}
  \frac{a(N_2) N_1^\frac12 N_2^{\frac12}}{N_4^{2}} \| u_{N_1}\|_{L^2}
  \| u_{N_2} \|_{L^2} \| u_{N_3} \|_{L^2} \| u_{N_4}\|_{L^2} \\
  = & \ \sum_{M \leq N_1 \leq N_2 \leq N_3 \sim N_4}\left( \frac{
      a(N_2) N_1}{a(N_1) N_2}\right)^{\frac12} \frac{N_2}{N_4} \|
  u_{N_1} \|_{H^a} \| u_{N_2} \|_{H^a} \| u_{N_3} \|_{\dot
    H^{-\frac12}} \| u_{N_4} \|_{\dot H^{-\frac12}}
\end{split}
\]
and the summation with respect to the $N_i$'s is now straightforward.
\end{proof}

We conclude the proof of Proposition~\ref{P:energybound}
with 

\begin{proof}[Proof of the estimate \ref{R6bd}]
Writing $\xi=\xi_4+\xi_5+\xi_6$ as the frequency decomposition in
the cubic product we write $R_6(u)$ in the form
\[
R_6(u) = \mp \frac14 \int_{P_6}i\xi b_4(\xi,\xi_1,\xi_2,\xi_3) 
\hat u(\xi_1) \hat u(\xi_2) \hat u(\xi_3)\hat u(\xi_4) \hat u(\xi_5)
\hat u(\xi_6) \, d\sigma \,.
\]
where $P_6 = \{ \xi_1+ \xi_2+\xi_3+ \xi_4 +\xi_5 +\xi_6 = 0\}$.
  For $b_4$ we use the extension given
by Lemma~\ref{L:b4bd}. Since this extension is smooth in all 
variables on the dyadic scales,  without any restriction we can separate
variables and reduce the problem to the case when $b_4$ is of  product
type.  Then we can return to the physical space and rewrite
\[
R_6(u) = \sum_{N,N_2,\cdots, N_7} N b_4(N,N_1,N_2,N_3)  \int 
u_{N_1} u_{N_2} u_{N_3} P_{N}( u_{N_4} u_{N_5} u_{N_6} ) dx, \quad 
u_{N_i} = P_{N_i} u
\]
where the factors in $b_4$ are harmlessly included in the spectral 
projectors. This is allowed because $L^2$ bounded multipliers are 
also bounded in $U^2_A$ spaces.

By symmetry we can assume that $N_1 \leq N_2 \leq N_3$,
as well as $N_4 \leq N_5 \leq N_6$.  We also take an increasing
rearrangement 
\[
\{ N_1,N_2,N_3,N_4,N_5,N_6\} = \{ M_1,M_2,M_3,M_4,M_5,M_6\}
\]
where we must always have $N \lesssim M_5 \sim M_6$.

Our next contention is that we can harmlessly discard the 
projector $P_N$ by separating variables. To see this we 
use the Fourier representation of the symbol
\[
p_N(\xi_1+\xi_2+\xi_3) = \int e^{i \lambda \xi_1} e^{i \lambda \xi_2}
 e^{i \lambda \xi_3} f(\lambda) d\lambda, \qquad  f_N(\lambda) = \int e^{-i \lambda \xi} p_N(\xi) d\xi
\]
The complex exponentials are bounded symbols and thus bounded on $U^2_A L^2$,
while $\| f_N\|_{L^1} \lesssim 1$ uniformly in $N$.

Assuming now that we have separated variables, we can
sum the coefficient in $R_6$ with respect to $N$ 
\[
\sum_{N \leq N_3} N b_4(N,N_1,N_2,N_3) \sim a(N_2) N_3^{-1} \lesssim a(M_2)
M_3^{-1}
\]
and we are left with having to estimate
\[
I = \sum_{M_1 \leq \cdots \leq M_5 = M_6} a(M_2) M_3^{-1} \int_{0}^1
\int_{\mathbb R} u_{M_1} u_{M_2} u_{M_3} u_{M_4} u_{M_5} u_{M_6} dx dt
\]
We divide the time interval $[0,1]$ in $M_6^{1-4s}$ subintervals of
size $M_6^{-1+4s}$ corresponding to the highest frequency factor. We
estimate the integral in each such subinterval, taking a loss of
$M_6^{1-4s}$ due to the interval summation.  Depending on how many
frequency $M_6$ factors there are we split into several cases:

\smallskip 
\noindent \textit{Case (a).} $M_4 \ll M_6$. Then we can use 
 two bilinear $L^2$ bounds for the products $u_3 u_5$ and $u_4 u_6$
and Bernstein to derive a pointwise bound for $u_{N_1}$ and $u_{N_2}$.
We obtain
\[
\begin{split}
|I_{(i)(a)} | \lesssim &  \sum_{M_1 \leq \cdots \leq M_5 = M_6} 
a(M_2) M_3^{-1} M_6^{1-4s}  M_6^{-2} M_1^{\frac12} M_2^\frac12
\sup_{|I| =M_6^{-1+4s}}  \prod_{j=1}^6  \|\chi_I u_{M_j}\|_{U^2_A L^2}
\\
 \lesssim & \sum_{M_1 \leq \cdots \leq M_5 = M_6} 
\!\! \left(\frac{a(M_2)M_1}{a(M_1) M_2}\right)^{\frac12}
  \frac{M_2}{M_3}  \frac{   M_6^{s}}{M_3^{s}}
 \frac{M_6^{s}}{M_4^{s}}   M_6^{-1-8s}
\prod_{j=1}^2  \|u_{M_j}\|_{X^a} \prod_{j=3}^6  \|u_{M_j}\|_{X^s_M}
\end{split}
\]
where the factors were reorganized to make clear the summation with
respect to the $M_j$'s. It is also transparent here that the total
balance of exponents can only be favorable if $s \geq -\frac18$.

\smallskip
\noindent \textit{Case (b).} $M_3 \ll M_4 \sim M_6$. The same
argument as above applies after observing that two of the frequencies 
$\xi_4$, $\xi_5$ and $\xi_6$ must have an $M_6$ separation,
therefore the bilinear $L^2$ estimate can be applied.

\smallskip
\noindent \textit{Case (c).} $ M_3 \sim M_6$.  As in the
previous case we can apply the $L^2$ bilinear estimate for two of the
high frequency factors, say $u_{M_5} u_{M_6}$.  Then we use the $L_x^4
L_t^\infty$ bound for $u_{M_2}$ and $u_{M_3}$, the $L_x^\infty L_t^2$
for $u_{M_4}$ as well as the $L^\infty$ bound for $u_{M_1}$. We obtain
\[
\begin{split}
  |I_{(i)(a)} | \lesssim & \sum_{M_1 \leq M_2 \leq M_3 = \cdots = M_6}
\!\!\!\!\!\!\!  
a(M_2) M_6^{-1} M_6^{1-4s} M_6^{-1} M_1^{\frac12} M_2^{\frac14}
  M_6^{-\frac34}\sup_{|I| =M_6^{-1+4s}}  \prod_{j=1}^6 
\|\chi_I u_{M_j}\|_{U^2_A L^2}
  \\
  \lesssim & \sum_{M_1 \leq M_2 \leq M_3 = \cdots = M_6} \!\!
\left(\frac{a(M_2)M_1}{a(M_1) M_2}\right)^{\frac12}
  \left(\frac{M_2}{M_3}\right)^\frac34   M_6^{-1-8s}
  \prod_{j=1}^2 \|u_{M_j}\|_{X^a} \prod_{j=3}^6
  \|u_{M_j}\|_{X^s_M}
\end{split}
\]
Again the summability with respect to $M_j$'s is straightforward.

\end{proof}

\section{Proof of Theorem \ref{T:main}}
\label{S:proof}

For expositional convenience, in this section, we will assume again
that we are in the more difficult case $s < 0$.  We first establish
a short time small data result:

\begin{proposition} \label{P:final}
  Let $M \geq  1$ and $-\frac14 \leq s < 0$. For any initial
  data $u_0\in \mathcal{S}$ with
\[
\|u_0\|_{H^s_M} \ll 1 \, ,
\]
the unique solution $u\in C([0,1]; \mathcal{S})$ to \eqref{E:mKdV}
(focusing or defocusing) satisfies
\[
\|u\|_{L_{[0,1]}^\infty H_M^s} \leq C \|u_0\|_{H^s_M} \, .
\]
\end{proposition}

\begin{proof}
  For $h \in [0,1]$ let $u_h$ be the global solution to \eqref{E:mKdV}
  with initial data $u_{0h}=hu_0$. By Lemma
  \ref{L:equationestimate} and the trilinear estimate (Prop.\
  \ref{P:trilinear}),
\begin{equation}
  \label{E:step1}
  \|u_h \|_{X^s_M} \lesssim 
\|u_h \|_{\ell^2L_{[0,1)}^\infty H_M^s} +  \|u_h \|_{X^s_M}^3 \,.
\end{equation}

By the energy bound  in Prop. \ref{P:energybound}
 we have
\begin{equation}
  \label{E:step2}
  \|u_h \|_{l^2 L_{[0,1)}^\infty H^s_M}^2   
\lesssim \|u_{h0}\|_{H^s_M}^2  +
\|u_h \|_{X^s_M}^4  + \|u_h \|_{X^s_M}^6 \,.
\end{equation}
Combining \eqref{E:step1} and \eqref{E:step2}, we obtain
\[
\|u_h \|_{X^s_M} \leq C( h \| u_0\|_{H_M^s} +
\|u_h \|_{X^s_M}^3)
\] 
Since $\| u_0\|_{H_M^s} \ll 1$ and $\|u_h \|_{X^s_M}$ is a continuous 
function of $h$ vanishing at $h=0$, we conclude via a continuity
argument that 
\[
\|u_h \|_{X^s_M} \lesssim  h\| u_0\|_{H_M^s}, \qquad h \in [0,1]
\]
Returning to \eqref{E:step2}, it follows that 
\[
\| u\|_{L^\infty H^s_M} \lesssim \|u_0\|_{ H^s_M}
\]
The proof is concluded.
\end{proof}

Given Proposition~\ref{P:final}, we can conclude the proof of 
 Theorem \ref{T:main} using a scaling argument. Let 
$0 \geq s > -\frac18$
and $u_0 \in H^s$ with $\| u_0 \|_{H^s} \leq R$. Then we have
\[
\|  u_0 \|_{H^{-\frac18}_{M}} \leq  R M^{-\frac18 - s}, \qquad M \geq 1
\]
Let $u_{0\lambda}(x) = \lambda u_0(\lambda x)$ and $u_\lambda(x,t) =
\lambda u(\lambda x,\lambda^3t)$.  Then $u_\lambda$ solves
\eqref{E:mKdV} with initial data $u_{0\lambda}$.  We consider
$u_\lambda$ on the time interval $[0,1)$, with $\lambda$ to be chosen
below. We have
\[
\|u_{\lambda 0}\|_{H^{-\frac18}_{\lambda M}} \leq 
\lambda^{\frac38}\|u_0\|_{H^{-\frac18}_M} \leq
 \lambda^{\frac38} R  M^{-\frac18 - s}, \qquad M, \lambda M > 1
\]
Taking $\lambda$ such that $ \lambda^{\frac38} R  M^{-\frac18 - s} \ll 1$
we can apply Proposition~\ref{P:final} to conclude that
\[
\|u_{\lambda}\|_{L^\infty_{[0,1]} H^{-\frac18}_{\lambda M}}  
\lesssim\|u_{0\lambda}\|_{ H^{-\frac18}_{\lambda M}}.  
\]
Scaling back to the interval $[0,T]$ with $T = \lambda^3$
we obtain
\[
\|u\|_{L^\infty_{[0,T]} H^{-\frac18}_{M}}  
\lesssim \|u_0\|_{H^{-\frac18}_{M}}   , \qquad T^{\frac18} R  M^{-\frac18 - s} \ll 1
\]
The last restriction gives a bound from below on $M$,
\[
M \gg M(R,T) \defeq  (R T^\frac18)^{(\frac18+s)^{-1}}
\]
Taking a weighted square sum with respect to such $M$
in the previous relation  we obtain
\[
\|u\|_{L^\infty_{[0,T]} H^{s}_{M(R,T)}}  
\lesssim \|u_0\|_{H^{s}_{M(R,T)}}   
\]
This in turn shows that 
\[
u\|_{L^\infty_{[0,T]} H^{s}}  
\lesssim M(R,T)^{-s}   \|u_0\|_{H^{s}}   
\]
concluding the proof of the theorem.

\end{document}